\documentclass{article}
\usepackage{amsmath}
\usepackage{amsfonts}
\usepackage{graphicx}
\usepackage{amssymb}
\usepackage{amsthm}
\usepackage{newlfont}
\usepackage[all]{xy}
 \newtheorem{thm}{Theorem}[section]
 \newtheorem{cor}[thm]{Corollary}
 \newtheorem{lem}[thm]{Lemma}
 \newtheorem{prop}[thm]{Proposition}
 \theoremstyle{definition}
 \newtheorem{defn}[thm]{Definition}
 
 \theoremstyle{remark}
 \newtheorem{rem}[thm]{Remark}
 \newtheorem{ex}[thm]{Example}
 \numberwithin{equation}{section}

 \newcommand{\Rat}{\operatorname{Rat}}

 \newcommand{\PGL}{\operatorname{PGL}}

 \newcommand{\chr}{\operatorname{char}}
 \newcommand{\Aut}{\operatorname{Aut}}

 \newcommand{\Poly}{\operatorname{Poly}}

\newlength{\defbaselineskip}
\begin{document}

\title{The McMullen Map In Positive Characteristic}

\author{Alon Levy}

\maketitle

\begin{abstract}
McMullen proved the moduli space of complex rational maps can be parametrized by the spectrum of all periodic-point multipliers up to a finite amount of data, with the well-understood exception of Latt\`{e}s maps. We generalize his method to large positive characteristic. McMullen's method is analytic; a modified version of the method using rigid analysis works over a function field over a finite field of characteristic larger than the degree of the map. Over a finite field with such characteristic it implies that, generically, rational maps can indeed be parametrized by their multiplier spectra up to a finite-to-one map. Moreover, the set of exceptions, that is positive-dimension varieties in moduli space with identical multipliers, maps to just a finite set of multiplier spectra. We also prove an application, generalizing a result of McMullen over the complex numbers: there is no generally convergent purely iterative root-finding algorithm over a non-archimedean field whose residue characteristic is larger than either the degree of the algorithm or the degree of the polynomial whose roots the algorithm finds.
\end{abstract}

\section{Introduction}

A rational function over a field $F$ can be studied by associating invariants to its fixed and periodic points. At a finite fixed point $z$ of a rational function $\varphi$, the
derivative $\varphi'(z)$ can be interpreted as the map on tangent spaces, and is conjugation-invariant: that is, if $A$ is any linear-fractional transformation,
then $\varphi'(z) = (A\varphi A^{-1})'(z)$. We call this conjugation-invariant derivative the \textbf{multiplier} of $\varphi$ at $z$. If $z$ is the point at
infinity, we can extend the definition of the multiplier by conjugating it to be any finite point. We can furthermore define the multiplier of a point of exact
period $n$ to be $(\varphi^{n})'(z)$; here and in the sequel, powers of maps denote iteration rather than multiplication. The set of multipliers of period $n$,
called the \textbf{multiplier spectrum}, depends only on the map $\varphi$, and because it is conjugation-invariant it really only depends on its conjugacy class.

Thus the symmetric functions in the multipliers are regular algebraic functions on the space of rational maps of degree $d$ modulo conjugation, $\mathrm{M}_{d}$.
Collecting the symmetric functions in the period-$n$ multipliers for each $n$, we obtain regular algebraic maps $\Lambda_{n}$ from $\mathrm{M}_{d}$ to large affine
spaces $\mathbb{A}^{K_{n}}$ where $K_{n}$ is the number of period-$n$ points. Because there are infinitely such maps, one for each $n$, a priori we would expect
there to be some large $n$ for which the map
$$\Lambda_{1} \times \ldots \times \Lambda_{n}: \mathrm{M}_{d} \to \mathbb{A}^{K_{1} + \ldots + K_{n}}$$ is injective. This would allow us to analyze the geometry
of $\mathrm{M}_{d}$ purely in terms of multipliers. The map is not injective, but we do have a partial result:

\begin{thm}\label{intromain}Over a field of characteristic $0$ or greater than $d$, for sufficiently large $n$ the map $$\Lambda_{1} \times \ldots \times
\Lambda_{n}: \mathrm{M}_{d} \to \mathbb{A}^{K_{1} + \ldots + K_{n}}$$ is finite-to-one outside a proper closed subset of $\mathrm{M}_{d}$.\end{thm}

Unfortunately, we only get generic finiteness. The following example provides a family of functions in $\mathrm{M}_{d}$ on which all periodic cycles have constant multipliers, which we call an
\textbf{isospectral family}:

\begin{defn}The map $\varphi: \mathbb{P}^{1} \to \mathbb{P}^{1}$ is called a \textbf{Latt\`{e}s map} if there is an elliptic curve $E$, a finite morphism $\alpha: E
\to E$, and a finite separable map $\pi$ such that the following diagram commutes:

\begin{equation*}
  \xymatrix@R+2em@C+2em{
  E \ar[r]^{\alpha} \ar[d]_{\pi} & E \ar[d]^{\pi} \\
  \mathbb{P}^{1} \ar[r]_{\varphi} & \mathbb{P}^{1}
  }
\end{equation*}

If we choose $\pi$ to be the projection by $P\sim -P$, then $\alpha$ must be of the form $\alpha_{m, T}: P \mapsto mP + T$ where $T \in E[2]$.\end{defn}

McMullen has proven the following result:

\begin{thm}\label{mcm}~\cite{McM} Excluding the Latt\`{e}s maps, there are no positive-dimension isospectral families over $\mathbb{C}$.\end{thm}

McMullen's result is stronger than Theorem~\ref{intromain} in characteristic zero. However, his proof uses inherently complex-analytic methods. Although it is true
over any characteristic-$0$ field by the Lefschetz principle, it does not give us any proof in positive characteristic. In fact, the result is false in
characteristic $p$, as shown in,

\begin{ex}\label{counter}Let $\varphi(z) = \psi(z^{p}) + az$ where $a$ is a constant and $\psi$ is a family of rational functions. The family is isospectral (all
period-$n$ multipliers at finite points are $a^{n}$). However, if $\psi$ is large enough, for example if $\dim\psi > 3$, then the family has no hope of reducing to
a point in $\mathrm{M}_{d}$ since $\mathrm{M}_{d} = \Rat_{d}/\PGL_{2}$ where $\Rat_{d}$ is the space of rational functions of degree $d$ and $\PGL_{2}$ acts by
coordinate change.\end{ex}

However, Example~\ref{counter} involves wild ramification. There are theoretical reasons to believe that a positive-characteristic version of Theorem~\ref{mcm} is
true for tamely ramified maps. Briefly, McMullen's theorem is intimately connected with Thurston's rigidity theorem. Specifically:

\begin{defn}A point $z$ is a \textbf{preperiodic point} of $\varphi$ if for some $n$, $\varphi^{n}(z)$ is a periodic point; we say the minimal such $n$ is the
\textbf{tail length}. Observe that $z$ is preperiodic if and only if it has finite forward orbit. If all of the critical points of $\varphi$ are preperiodic, we say
that $\varphi$ is \textbf{postcritically finite}, or in short PCF.\end{defn}

There are $2d-2$ critical points counted with multiplicity when $\varphi$ is tamely ramified, and so the condition that $\varphi$ is PCF, at least if we fix the
cycle and tail length of each critical point, is a set of $2d-2$ algebraic equations. Since $\dim\mathrm{M}_{d} = 2d - 2$, there should be finitely many PCF maps
for each degree, cycle length, and tail length. Over $\mathbb{C}$, Thurston has proven that this is more or less the case:

\begin{thm}\label{rigidity}\cite{BBLPP, DH} (Thurston's Rigidity) Excluding the Latt\`{e}s maps, all of which are PCF, there are no positive-dimension PCF
families.\end{thm}

\begin{rem}It is expected that the PCF equations meet transversally, i.e. fixing a cycle and tail length, the scheme of non-Latt\`{e}s PCF maps is a finite set of reduced points; see Epstein's work in Section 5 of~\cite{Eps3}. Favre and Gauthier~\cite{FG} also show this explicitly for some relations involving PCF polynomials.\end{rem}

It is not a coincidence that the Latt\`{e}s maps form the sole counterexample to both Thurston's rigidity and McMullen's theorem. McMullen's proof heavily uses
rigidity. Briefly, his method in~\cite{McM} is to assume an isospectral family exists, and then label an infinite set of so-called \textbf{repelling points}, that
is periodic points whose multiplier $\lambda$ satisfies $|\lambda| > 1$, accumulating at a point in the Julia set. A crucial fact about isospectral families is that
their periodic points move without collision with one another or with critical points. This labeling cannot be done globally, but can be done locally, using Ma\~{n}\'{e}-Sad-Sullivan stability theory. Using previous results showing that this infinite set moves holomorphically,
he shows that a critical point cannot pass through a point that is preperiodic to a repelling point; since there are infinitely many repelling points over
$\mathbb{C}$, this implies that in an isospectral family all critical points are preperiodic, and then Thurston's rigidity shows that such a family is Latt\`{e}s or
trivial.

We will follow a similar program. Many of the methods of complex analysis have been successfully ported to the non-archimedean setting, and although a full version
of Thurston's rigidity is beyond our reach, a sufficiently good approximation is proven in~\cite{Lev3}. First, we have,

\begin{defn}\label{tame}Let $\varphi$ be defined over a complete algebraically closed non-archimedean field $K$. We say that $\varphi$ is \textbf{tame} if the local
degree at any open analytic disk is not divisible by the residue characteristic $p$ of $K$. Equivalently, $\varphi = f/g$ is tame if and only if for each integral model
of $\varphi$ over $\mathcal{O}_{K}$, the reduction mod the maximal ideal is tamely ramified after clearing common factors of the reductions of $f$ and $g$. If
$\varphi$ is defined over a global field, we say it is tame if it is tame when we regard it as a map over every completion.\end{defn}

\begin{rem}Whenever $\deg\varphi < p$, or whenever $\varphi$ is the composition of maps of degree less than $p$, $\varphi$ is tame.\end{rem}

We have as a prior result~\cite{Lev3},

\begin{thm}\label{auto}Let $\varphi(z) \in F(z)$ where $F$ is a global function field such that $\chr F = p$. If $\varphi$ is PCF and tame at every place of $F$, then the
multipliers of $\varphi$ are all defined over $\overline{\mathbb{F}_{p}}$.\end{thm}

\begin{rem}An immediate corollary of Theorem~\ref{auto} is that if we can prove McMullen's theorem for a family, then we can prove the finiteness portion of
Thurston's rigidity for it as long as the local degrees are not divisible by $p$. In~\cite{Sil96} it is proven that $\Lambda_{1}$ is an isomorphism over
$\mathbb{Z}$ from $\mathrm{M}_{2}$ to a plane in $\mathbb{A}^{3}$, and thus Thurston's rigidity is also true for quadratic maps in characteristic $0$ or $p \geq
3$.\end{rem}

Like McMullen, we will prove that in an isospectral family we can label a set of periodic points that accumulate at a point that is preperiodic to a repelling point, and that if a critical point passes through the preperiodic point it is equal to it for all maps in the family. We do not have a non-archimedean Ma\~{n}\'{e}-Sad-Sullivan stability theory, but we do have a way to locally label an infinite set of repelling points, using work of Baldassarri~\cite{Bal}. This will suffice to show that an isospectral map is PCF. This requires a repelling point in $\mathbb{P}^{1}_{K}$ where $K$ is a complete algebraically closed characteristic-$p$ field; unlike in the complex case, a repelling point is not guaranteed to exist. However, if we start from a global field $F$, if the multipliers are not all in $\overline{\mathbb{F}_{p}}$ then we can take the completion with respect to a valuation that occurs in the denominator of a multiplier.

We will prove,

\begin{thm}\label{main}Let $F$ be a global function field of characteristic $p$. Suppose that $C$ is an isospectral one-parameter family of rational functions defined over $F$ all of which are tame. Suppose further that the multipliers are not all contained in $\overline{\mathbb{F}_{p}}$. Then $C$ is trivial: that is, all maps in $C$ are conjugate to one another, and $C$ maps down to a single point of $\mathrm{M}_{d}$.\end{thm}

\begin{rem}The Latt\`{e}s family has multipliers defined over $\overline{\mathbb{F}_{p}}$, and so is automatically excluded from the hypotheses of the
theorem.\end{rem}

Let us now prove Theorem~\ref{intromain} assuming Theorem~\ref{main}. We will prove the slightly stronger corollary:

\begin{cor}\label{maincor}Let $k$ be an algebraically field of characteristic greater than $d$. For sufficiently large $n$, there are only finitely many spectra in $\mathbb{A}^{K_{1} + \ldots + K_{n}}_{k}$, including the Latt\`{e}s multiplier spectra, for which there is a positive-dimension family of morphisms in $\mathrm{M}_{d}(k)$ with these spectra.\end{cor}

\begin{proof}Let $X_{n}$ be the union of all positive-dimension varieties in $\mathbb{A}^{K_{1} + \ldots + K_{n}}_{k}$ on which the map $\Lambda_{1} \times \ldots \times \Lambda_{n}$ has positive relative dimension. Note that if $m < n$, the map $\Lambda_{1} \times \ldots \times \Lambda_{m}$ commutes with the composition of $\Lambda_{1} \times \ldots \times \Lambda_{n}$ with the projection obtained by forgetting all period-higher-than-$m$ multipliers. This gives us maps from $X_{n}$ to $X_{m}$, commuting with the multiplier spectrum maps on $\mathrm{M}_{d}$. We need to prove that for sufficiently large $n$, $X_{n}$ is empty. Suppose to the contrary that $X_{n}$ is nonempty for all $n$. For sufficiently large $n$, the generic relative dimension of $\Lambda_{1} \times \ldots \times \Lambda_{n}$ on the preimage of $X_{n}$ stabilizes; therefore, for sufficiently large $n$ and $m$, with $m < n$, the map $X_{n} \to X_{m}$ is generically one-to-one. Informally, it means that for sufficiently large $n$, finding more multipliers will not give us any more information.

Now, label an infinite sequence of curves $Y_{i}$ contained in $X_{i}$, where $i \geq n$, such that the curves are all birationally equivalent and, if $i > j$, $Y_{i}$ maps to $Y_{j}$ under the forgetful map on multiplier spectra. Let $Z$ be a subvariety of the common preimage of all the $Y_{i}$s in $\mathrm{M}_{d}$ that maps to each $Y_{i}$ with relative dimension $1$. We can consider $Z$ to be a one-parameter family over a finite extension of $F = k(Y_{n})$, which is isospectral, mapping to a single point of $\mathbb{A}^{K_{1} + \ldots + K_{n}}_{F}$. Clearly, the multipliers are not contained in $\overline{\mathbb{F}_{p}}$. This contradicts Theorem~\ref{main}, which claims that $Z$ must be the trivial family.\end{proof}

Theorem~\ref{intromain} follows as soon as we exhibit two maps with two different multiplier spectra, proving that $\Lambda_{1} \times \ldots \times \Lambda_{n}$ doesn't collapse all of $\mathrm{M}_{d}$ to a single point. Finding such a pair of maps is easy: for example, $z^{d} + 1$ has just one fixed point that is also critical whereas $z^{d}$ has two, and fixed points are critical precisely when their multipliers are $0$.

The result is not constructive: it does not tell us what $n$ we need to choose to ensure that the map $\Lambda_{1} \times \ldots \times \Lambda_{n}$ will be
finite-to-one, nor does it tell us the degree of such a map. Unlike in McMullen's characteristic $0$ case, Theorem~\ref{intromain} does not even tell us what the
exceptions are to finiteness, i.e. what the finitely many spectra with positive-dimension preimages in Corollary~\ref{maincor} are. For some partial results in that
direction, see~\cite{HT, Sil96}.

Effective versions of Theorem~\ref{intromain} would require different techniques. Unlike the techniques in~\cite{HT}, we are unable to list the finite set of
exceptions in Corollary~\ref{maincor} (or any proper subset of $\mathrm{M}_{d}$ that contains all of them), although we conjecture that the Latt\`{e}s family is the
only one. Nor can we use these techniques to find or even bound the degree of the McMullen map, or the minimal period $n$ such that $\Lambda_{1} \times \ldots
\times \Lambda_{n}$ is generically finite.

We spend Section~\ref{outline} proving Theorem~\ref{main}.This proof is easier than in the complex case, since non-archimedean analysis is more rigid than complex analysis, making it easier to show that a critical point cannot move very freely in an isospectral family.

We give an application of the result in Sections~\ref{algorithms}. McMullen originally proved Theorem~\ref{mcm} in order to study generally convergent purely iterative root-finding algorithms. Following McMullen, we make the following definition:

\begin{defn}\label{iterdef}A \textbf{purely iterative root-finding algorithm} is a rational map $T: \Poly_{r} \dashrightarrow \Rat_{d}$, where $\Poly_{r} \subseteq \Rat_{r}$ is the space of degree-$r$ polynomials. Let us denote the image of a polynomial $f(z)$ under $T$ by $T_{f}(z)$. We say that $T$ is \textbf{generally convergent} if, for an open dense set of $f \in \Poly_{r}$, the sequence $T^{n}_{f}(z)$ converges to a root of $f$ on an open dense subset of $\mathbb{P}^{1}$.\end{defn}

The definition of generally convergent algorithms is made relative to a fixed topology. It is a question of Smale~\cite{Sma} whether such algorithms exist in the complex setting. McMullen proved that they do not when $r \geq 4$ in the complex setting, by showing that the image of a generally convergent algorithm, $T(\Poly_{r}) \subseteq \Rat_{d}$, is in fact an isospectral family, hence constant. We use an amplification of Theorem~\ref{main} and an elementary computation to prove,

\begin{thm}\label{iterative}There is no generally convergent purely iterative root-finding algorithm $T: \Poly_{r} \dashrightarrow \Rat_{d}$ over a non-archimedean field of residue characteristic $p$ if $p = 0$, $p > d$, or $p \geq r$.\end{thm}

\section{Proof}\label{outline}

Let $C$ be an isospectral irreducible curve as in the hypotheses of Theorem~\ref{main}. Since the multipliers of $C$ are not defined over $\overline{\mathbb{F}_{p}}$, there is a place of $F$ such that one multiplier is repelling. Let $K$ be an algebraically closed field, complete with respect to a nontrivial non-archimedean absolute value $|\cdot|$; we choose $K$ to be an algebraically closed completion of $F$ such that one of the multipliers of the family $C$ is repelling. By a result of B\'{e}zivin~\cite{Bez}, this means that there are infinitely many repelling points. Moreover, the Julia set of $\varphi$ is the closure of the set of repelling points. B\'{e}zivin's paper assumes characteristic $0$, but the proofs and the result port perfectly to positive characteristic as long as the map is tamely ramified.

Moreover, repelling points move without collision: that is, if we mark any number of them, they will stay distinct, since a fixed point is repeated iff it has
multiplier $1$. McMullen's proof in fact hinges on using this to show that the repelling points, and with them the entire Julia set, move rigidly, so that after a
local analytic conjugation the Julia set is constant in the family.

We now mark all critical points and enough other special points (such as critical values or periodic points) to obtain an enhanced moduli space $\mathrm{M}_{d}(\Gamma)$ so that there are no nontrivial automorphisms of
any map in the family $C' = C \times_{\mathrm{M}_{d}} \mathrm{M}_{d}(\Gamma)$. The natural projection $\mathrm{M}_{d}(\Gamma) \to \mathrm{M}_{d}$ is finite-to-one,
and by abuse of notation we denote the multiplier map on $\mathrm{M}_{d}(\Gamma)$ by $\Lambda_{n}$, sending each enhanced $\varphi$ to an ordered set of the
multipliers of all the marked period-$n$ points and the spectrum of the multipliers of the unmarked points. For example, if we mark all critical points and values,
then there are automorphisms if and only if there are exactly $2$ distinct ones, which is the case only if $\varphi$ is conjugate to $z^{\pm d}$, in which case the
multipliers lie in $\mathbb{F}_{p}$ and the map is excluded from the hypotheses of Theorem~\ref{main}.

The upshot is that,

\begin{lem}\label{unr}Suppose that $C'$ is chosen in such a way that for all $\varphi \in C'$, $\Aut\varphi$ is trivial. Let $C'' = C'
\times_{\mathrm{M}_{d}(\Gamma)} \mathrm{M}_{d}(\Gamma')$ where $\Gamma'$ enhances $\Gamma$ with one additional repelling periodic point. Then the natural projection
$\pi: C'' \to C'$ is unramified.\end{lem}

\begin{proof}First, let $n$ be the period, and let $f$ be the polynomial whose roots are the period-$n$ points that are unmarked in $\Gamma$. The repelling points
are simple roots of $f$ since repeated roots have multiplier $1$, so ramification can only occur if there exists some automorphism of $\varphi$ mapping $z_{1}$ to
$z_{2}$ where $z_{1}$ and $z_{2}$ are both roots of $f$. Since $\varphi$ has no automorphisms preserving the enhanced data of $\Gamma$, there is no
ramification.\end{proof}

Now, let us change coordinates so that $0$ is persistently preperiodic to a repelling point. Note that the family parametrizing $\varphi$ is trivial if and only if the family parametrizing $\varphi^{n}$ is trivial for any (or all) $n$; this appears well-known to dynamicists, and is written down in Proposition 2.1 of~\cite{Lev4}. Therefore, we may apply iteration and assume that $\varphi(0)$ is fixed, and for simplicity we will conjugate so that $\varphi(0) = \infty$. Since $0$ is in the Julia set, there is a sequence of repelling periodic points converging to it for
every map in the family.

Recall now the following results from non-archimedean analysis:

\begin{prop}\label{nonarch}Let $C$ be a curve. Let $x \in C$, and let $D \ni x$ be a neighborhood of $x$ in $C$ that is analytically isomorphic to a disk in
$\mathbb{P}^{1}$. The following functions are analytic on $D$:

\begin{enumerate}
\item The restriction to $D$ of any morphism $f: C \to X$ for any curve $X$.\\
\item Any $g$ from $D$ to $C'$ where $f: C' \to C$ is a rational map and $f\circ g$ is the identity on $D$ and $f$ maps $g(D)$ to $D$ with local degree $1$.
\end{enumerate}\end{prop}

Observe also that if in a projection map from $C'$ to $C$ a point $x \in C$ is unramified, there exists a neighborhood of $x$ in $C$ with a preimage disk that maps
onto it with degree $1$, so that we can apply the second case of Proposition~\ref{nonarch} and obtain a local inverse. We have,

\begin{lem}\label{hope}Let $C'$ and $C''$ be as in the hypotheses of Lemma~\ref{unr}, and let $x \in C'$. There exists a disk $D \ni x$ in $C'$ that is analytically
isomorphic to a $\mathbb{P}^{1}$-disk, depending only on $C'$, with the following property: the preimage of $D$ in $C''$ is a disjoint union of
$\mathbb{P}^{1}$-disks each mapping into $D$ with degree $1$.\end{lem}

\begin{proof}First, we need to restrict to disks that miss the singular points of $C'$; these are of course uniform. Now, Theorem 0.2 in~\cite{Bal} constructs the radius of $D$ explicitly.\end{proof}

Observe that as a trivial consequence of joining Lemma~\ref{hope} and Proposition~\ref{nonarch}, we obtain,

\begin{cor}\label{inflab}Suppose that $0$ is a persistently Julia point. For each point $\varphi \in C'$, there exists an infinite set of repelling points $z_{i}$
that converge to $0$, such that for some neighborhood of $D \ni \varphi$, the functions $z_{i}$ are all analytic on $D$. Moreover, the convergence of $z_{i}$ is
uniform on $D$.\end{cor}

\begin{proof}By repelling density, there is a sequence of repelling points $z_{i}$ converging to $0$. By Lemma~\ref{hope}, the hypotheses of
Proposition~\ref{nonarch} are satisfied and each repelling point is an analytic map from $D$ to a neighborhood $D'$ in $C''$. Now the map $\theta$ from $C''$ to
$\mathbb{P}^{1}$ defined by $z_{i}$ is rational, and again using Proposition~\ref{nonarch} it is analytic. Now, regardless of whether $0$ is periodic, it cannot
collide with any of the points $z_{i}$ for any map in $D$. Thus, $z_{i}(D) \neq 0$ and so since $z_{i}$ is analytic it maps $D$ into a disk $D(a_{i}, r_{i})$ where
$r_{i} < |a_{i}|$ and since the absolute value is constant on each such disk we have $a_{i} \to 0$ and thus $z_{i} \to 0$ uniformly on $D$.\end{proof}

An infinite labeled sequence of analytic functions converging to $0$ is the best we can hope for. We use the non-collision of periodic points to further prove

\begin{lem}\label{persist}Suppose that a critical point $x$ is preperiodic to a repelling cycle for some point on $C'$. Then it is preperiodic, with the same cycle
and tail lengths, for all points on $C'$.\end{lem}

\begin{proof}After iteration, we may assume that $\varphi(x)$ is a fixed point. We will show that this remains true under infinitesimal perturbation; since the
property of mapping to a fixed point is algebraic, cut out by the single equation $\varphi^{2}(x) = \varphi(x)$, it holds either for finitely many points on $C'$ or
on all of $C'$.

We assume that $\infty$ is repelling and that $0$ maps to $\infty$ for all $\varphi \in C'$. Because we marked the critical points on $C'$, the function $x$ is
rational on $C'$, and its image misses $\infty$, and so it maps each disk in $C'$ to a disk analytically. Suppose that for some $\varphi_{0} \in C'$, $x = 0$, and
take a neighborhood of $D \ni \varphi_{0}$ that is isomorphic to a $\mathbb{P}^{1}$-disk, such that $x$ becomes an analytic function on $D$. Let us conjugate $D$ to
$D(0, 1)$ and $\varphi_{0}$ to $0 \in D(0, 1)$, so that $x(0) = 0$, and denote the map corresponding to $c \in D(0, 1)$ by $\varphi_{c}$.

Now, by Corollary~\ref{inflab}, there exists an infinite sequence of analytic functions $z_{i}$ on $D$ such that $z_{i}(c)$ is a repelling periodic point for
$\varphi_{c}$ and $z_{i}(c) \to 0$ uniformly. From the proof of Corollary~\ref{inflab}, $|z_{i}(c)| = |z_{i}(0)|$ for all $i$ and $c \in D(0, 1)$.

Moreover, $z_{i}(c) - x(c)$ is an analytic function that is never zero: the multiplier at $z_{i}$ is not zero, which makes it impossible for it to collide with a
critical point. Now $z_{i}(c) - x(c)$ maps over a disk that includes $z_{i}(0)$ and excludes $0$, so that $|z_{i}(c) - x(c)| \leq |z_{i}(0)|$.

Finally, we obtain $$|x(c)| \leq \max\{|z_{i}(0)|, |z_{i}(c)|\} = |z_{i}(0)| \to_{i \to \infty} 0$$ and so, $x(c) = 0$ on $D$, hence on all of
$C'$.\end{proof}

The next step is almost identical to the complex case:

\begin{lem}\label{preper}All critical points of all maps in $C$ are persistently preperiodic, with uniformly bounded cycle and tail lengths.\end{lem}

\begin{proof}We consider each critical point separately. If it is periodic, then it is persistently so because $C'$ is isospectral, and the period is constant. If it
is preperiodic to a repelling cycle, then the same is also true by Lemma~\ref{persist}. Let us now assume it is not preperiodic to a repelling cycle and derive a
contradiction. There is a rational function $x$ from $C'$ to $\mathbb{P}^{1}$ sending a map to a critical point, and then the functions $\varphi^{i}(x)$ are
rational for all $i$. Those functions all miss the repelling points, so if we fix any three repelling points, they become functions from $C'$ to $\mathbb{P}^{1}$
minus three points. Recall that,

\begin{lem}\label{mono}There are finitely many nonconstant separable maps from an affine algebraic curve to $\mathbb{P}^{1}$ minus three points.\end{lem}

\begin{proof}It suffices to pass to the normalization. Write the algebraic curve as $X$ minus $x_{1}, \ldots, x_{n}$, where $X$ is smooth. We need to show there are finitely many maps from $X$ to $\mathbb{P}^{1}$ such that the
preimage of $\{0, 1, \infty\}$ is contained in $\{x_{1}, \ldots, x_{n}\}$. For each degree $d$ there are finitely many possible maps, determined by the preimages of
$(0, 1, \infty)$ and their multiplicities; this is because knowing the zeros and poles (counted with multiplicity) of a rational function on an algebraic curve is enough to determine the map up to multiplication by a constant. It remains to bound $d$. Now let
$e_{i}$ be the multiplicity with which $x_{i}$ maps. We have $\sum e_{i} \geq 3d$ and $\sum(e_{i} - 1) \leq 2d - 2 + 2g(X)$, which gives us $n \geq d + 2 - 2g(X)$,
valid as long as the maps are separable, even if they are wildly ramified.\end{proof}

Now, the function $x$ is separable (in fact tamely ramified) since $\varphi$ is tamely ramified, forcing the local degree to not be divisible by $p$. Furthermore,
$\varphi^{i}(x)$ is also separable, since $\varphi$ is separable. We can now apply Lemma~\ref{mono}. Since the sequence of functions $\varphi^{i}(x)$ comes from a
finite list, two functions (say $i = m, n$) have to coincide, making $x$ persistently preperiodic after all, and the tail and cycle lengths are again fixed by the
choice of $m$ and $n$. Alternatively, all but finitely many functions have to be constant, and then we know the values of $\varphi$ at infinitely many given points; knowing the values of $\varphi$ at $2d+1$ distinct points is enough to determine $\varphi$, and therefore $C$ is trivial, contradicting the assumption that it is nontrivial.\end{proof}

The final step is to note that tame PCF maps over function fields have multipliers lying in $\overline{\mathbb{F}_{p}}$; this is Corollary $1.7$ of~\cite{Lev3}. Since by assumption this is not the case for $C'$, we have a contradiction, and Theorem~\ref{main} is proved.

\section{Iterative Root-Finding Algorithms}\label{algorithms}

Iterative root-finding algorithms are techniques used to find the roots of a polynomial by using iteration; see Definition~\ref{iterdef}. To prove that they do not exist in most non-archimedean cases, we first amplify Theorem~\ref{main}:

\begin{prop}\label{almostiso}Let $F$ be a global function field of characteristic $p$. Suppose that $C$ is a family of rational functions defined over $F$ all of which are tame. Suppose further that all but finitely many of the multiplier spectra of $C$ are constant, and that those constant multipliers are not all contained in $\overline{\mathbb{F}_{p}}$. Then $C$ is trivial.\end{prop}

\begin{proof}The proof in Section~\ref{outline} works verbatim. We pick a completion $K$ such that one of those constant-multiplier cycles is repelling, and iterate to obtain a repelling fixed point, which we fix coordinates to put at $\infty$. We also make $0$ one of its preimages. We again apply the result in~\cite{Bez} to show that there's an infinite sequence of repelling points accumulating at $0$; all but finitely many of these points have constant multipliers in $C$, and therefore we are guaranteed to have a subsequence of repelling points, with constant multipliers, accumulating at $0$. From this point we immediately obtain Lemma~\ref{hope}. We also obtain a modification of Lemma~\ref{persist}, excluding the finitely many repelling points with nonconstant multipliers; this is enough to obtain three repelling points that critical points must miss if they are not to be persistently preperiodic, which is enough to prove Lemma~\ref{preper}.\end{proof}

\begin{rem}In McMullen's original formulation, the families $C \subseteq \mathrm{M}_{d}$ that were proven to be constant or Latt\`{e}s were slightly broader than isospectral ones: McMullen studied so-called stable families, that is, families for which the period of the longest attracting cycle is bounded. Of course, for cycle lengths beyond this uniform bound, the multipliers miss infinitely many values (all attracting ones), and therefore are constant in any algebraic family. Thus, McMullen's stable families meet the condition in Proposition~\ref{almostiso} that all but finitely many multiplier spectra be constant. They trivially meet the other conditions: all functions over $\mathbb{C}$ are tame, and the condition that the multipliers not all be contained in $\overline{\mathbb{F}_{p}}$ really exists only to obtain a repelling cycle in some completion $K_{v}$ of $F$, whereas by~\cite{Shi}, all but at most $2d-2$ cycles of any map in $\mathrm{M}_{d}(\mathbb{C})$ are repelling.\end{rem}

We now can prove Theorem~\ref{iterative}.

\begin{proof}[Proof of Theorem~\ref{iterative}]If $p \geq r$ or $p = 0$, then the proof is not a corollary of the results in the remainder of this paper. However, if $p > d$ then it is a corollary of Proposition~\ref{almostiso}; we prove the theorem under either condition.

If $T^{n}_{f}(z)$ converges to a root of $f$ on an open dense subset of $\mathbb{P}^{1}$, then it is impossible for $T_{f}$ to have non-repelling cycles except those corresponding to the roots of $f$. To see why, observe that attracting cycles have open basins of attraction, and over a non-archimedean field indifferent cycles have open basins on which points are recurrent; see section $3$ of~\cite{RL}. Attracting fixed points that are not roots of $f$ will then have an open set converging to them, and indifferent points and attracting or indifferent cycles of period more than $1$ have basins on which points either are already periodic or have orbits that do not converge to anything.

\begin{rem}Periodic points may occur near an indifferent fixed point $z$ if its multiplier is a root of unity, and we work over a field $K$ of positive characteristic. If $\chr K = 0$ or the multiplier at $z$ is not a root of unity, then it is proven in~\cite{RL} that we can find an open set containing $z$ without periodic points other than $z$; see also~\cite{LRL}. It is conjectured that we can also find such an open set if $\chr K > 0$ and the multiplier at $z$ is a root of unity. However, even if this conjecture is false, it is certainly true that near an indifferent fixed point $z$, points have orbits that do not converge to another attracting fixed point.\end{rem}

We will deal with the $p \geq r$ and $p > d$ cases separately. First, let us assume that $p \geq r$ or $p = 0$. When the fixed points are all distinct, that is, when none of their multipliers is $1$, the following formula holds for their multipliers $\lambda_{i}$: $$\sum_{i = 1}^{d+1}\frac{1}{1 - \lambda_{i}} = 1$$ Let us now look at the formula modulo the maximal ideal $\mathfrak{m}$ of the ring of integers of the field $K$. Whenever $\lambda_{i}$ corresponds to an attracting point, we have $1/(1-\lambda_{i}) \equiv 1 \mod \mathfrak{m}$, and whenever it corresponds to a repelling point, $|1-\lambda_{i}| > 1$ and so $1/(1-\lambda_{i}) \equiv 0 \mod \mathfrak{m}$. Since we have $r$ attracting points, we have $r \equiv 1 \mod \mathfrak{m}$. This implies $p | r-1$, which contradicts the assumption that $p \geq r$ or $p = 0$.

Let us now assume that $p < r$, but still $p > d$. Observe that this implies that $r \geq 4$. The image $T(\Poly_{r}) \subseteq \Rat_{d}$ is an algebraic family, and $\Lambda_{n}$ applied to $T(\Poly_{r})$ is also algebraic, so if it misses infinitely many values of each multiplier (namely, the attracting ones), it must be constant. This shows that the multiplier spectra on $T(\Poly_{r})$ are constant except possibly for the finitely many corresponding to fixed points. Moreover, the attracting cycles of period more than $1$ must be repelling, so clearly the multipliers are not in $\overline{\mathbb{F}_{p}}$.

We apply Proposition~\ref{almostiso} and obtain that $T(\Poly_{r})$ maps to a single point of $\mathrm{M}_{d}$, which we call $\varphi$. An open dense set of $\mathbb{P}^{1}$ converges to a finite set of points, called the sinks of $\varphi$, which must be contained in the set of roots of every $f \in \Poly_{r}$. Moreover, the set of sinks of $\varphi$ must coincide with the set of roots of $f$, because the coefficients of $T_{f}$ depend algebraically on the coefficients of $f$, which are symmetric functions of the roots of $f$. This means that the existence of a generally convergent algorithm implies that an open dense set of $f \in \Poly_{r}$ have identical configurations of roots in $\mathbb{P}^{1}$. Since $r \geq 4$, this contradicts the existence of the cross-ratio.\end{proof}

The portion of the proof assuming $p \geq r$ is unique to the non-archimedean setting, while the portion assuming $p > d$ generalizes McMullen's work. As a result, McMullen only shows there do not exist generally convergent algorithms if $r \geq 4$. Over the complex numbers, there exist generally convergent algorithms for polynomials of degrees $2$ and $3$. In degree $2$, Newton's method is such an algorithm. Over a non-archimedean field, Newton's method, equivalent to the rational function $\varphi(z) = z^{2}$, is no longer generally convergent. Over the complex numbers, the open dense set lying away from the unit circle has orbit converging to $0$ and $\infty$, but in any non-archimedean field, the equivalent of the unit circle is the set of finite nonzero residue classes, which has nonempty interior.

\bibliographystyle{amsplain}
\bibliography{rigidity}

\bigskip\noindent \sc{Alon Levy -- Department of Mathematics, KTH, Stockholm, Sweden}

\noindent \tt{email: alonlevy@kth.se}

\end{document}